\documentclass[11pt, british]{amsart}

\usepackage{babel,eucal,url,amssymb,
enumerate,amscd,
}
\usepackage{color}
\usepackage{framed}

 \newtheorem{thm}{Theorem}[section]
 \newtheorem{cor}[thm]{Corollary}
 \newtheorem{lem}[thm]{Lemma}
 \newtheorem{prop}[thm]{Proposition}
 \newtheorem{defn}[thm]{Definition}
 \newtheorem{rem}[thm]{Remark}
 
 \numberwithin{equation}{section}

\begin{document}

\title[] {Slant and Legendre Null  Curves in 3-Dimensional Sasaki-like Almost
 Contact  B-Metric  Manifolds}

    \author{Galia Nakova}
    \address{Galia Nakova, Affiliation}    \curraddr{University of Veliko Tarnovo "St. Cyril and St. Methodius" \\ Faculty of Mathematics and Informatics
\\ 2 Teodosii Tarnovski Str., Veliko Tarnovo 5003, Bulgaria}
    \email{gnakova@gmail.com}
    
    \author{\fbox{Simeon Zamkovoy}}
 
%

\keywords{Sasaki-like almost contact B-metric manifolds, Slant curves, Null curves}



\subjclass{53C15, 53C50}


In memory of Professor Simeon Zamkovoy, \\ University of Sofia, Bulgaria

\begin{abstract}
Object of study in the present paper are slant and Legendre null curves in 3-dimensional Sasaki-like almost contact B-metric manifolds. For the examined curves we express the general Frenet frame and the Frenet frame  for which the original parameter is distinguished, as well as the corresponding curvatures, in terms of the structure on the manifold. We prove that the curvatures  of a framed null slant and Legendre curve are constants if and only if a specific function for the considered  manifolds is a constant. We find a necessary and sufficient condition a slant null curve to be a generalized helix and a Legendre null curve to be a null cubic. For some of investigated curves we show that they are non-null slant or Legendre curves
with respect to the associated B-metric on the manifold. We give examples of the examined curves. 
Some of them are constructed in a 3-dimensional Lie group as Sasaki-like manifold and  their matrix representation is obtained.
\end{abstract}

\newcommand{\ie}{i.\,e. }
\newcommand{\g}{\mathfrak{g}}
\newcommand{\D}{\mathcal{D}}
\newcommand{\F}{\mathcal{F}}
\newcommand{\diag}{\mathrm{diag}}
\newcommand{\End}{\mathrm{End}}
\newcommand{\im}{\mathrm{Im}}
\newcommand{\id}{\mathrm{id}}
\newcommand{\Hom}{\mathrm{Hom}}

\newcommand{\Rad}{\mathrm{Rad}}
\newcommand{\rank}{\mathrm{rank}}
\newcommand{\const}{\mathrm{const}}
\newcommand{\tr}{{\rm tr}}
\newcommand{\ltr}{\mathrm{ltr}}
\newcommand{\codim}{\mathrm{codim}}
\newcommand{\Ker}{\mathrm{Ker}}
\newcommand{\R}{\mathbb{R}}
\newcommand{\E}{\mathbb{E}}
\newcommand{\K}{\mathbb{K}}

\newcommand{\thmref}[1]{Theorem~\ref{#1}}
\newcommand{\propref}[1]{Proposition~\ref{#1}}
\newcommand{\corref}[1]{Corollary~\ref{#1}}
\newcommand{\secref}[1]{\S\ref{#1}}
\newcommand{\lemref}[1]{Lemma~\ref{#1}}
\newcommand{\dfnref}[1]{Definition~\ref{#1}}


\newcommand{\ee}{\end{equation}}
\newcommand{\be}[1]{\begin{equation}\label{#1}}

\maketitle

\section{Introduction}\label{sec-1}
In the Lorentzian geometry there exist three types of curves according to the causal character of their tangent vector -- spacelike, timelike and null (lightlike) curves. Studying the geometry of null curves is of special interest since they have very different properties compared to spacelike and timelike curves. 
The general theory of null curves is developed in \cite{D-B, D-J},
where there are established important applications of these curves in general relativity. \\
A slant curve in almost contact and almost paracontact geometry arises as a generalization of a curve of constant slope (also called a cylindrical helix) in an Euclidean space $\E^3$. More precisely, a slant curve $C(t)$ is defined by the condition the scalar product $g(\dot C(t),\xi )$ of the tangent vector field $\dot C(t)$ and the Reeb vector field $\xi $ to be a constant. Therefore, slant curves can be also viewd as a generalization of Legendre curves for which $g(\dot C(t),\xi )=0$. Slant and Legendre curves in almost contact metric and almost paracontact metric manifolds have been investigated  intensively by many  authors \cite{I-L, W} and the references therein.\\ 
In this paper we consider slant and Legendre null curves in 3-dimensional Sasaki-like almost contact B-metric manifolds, which are Lorentzian manifolds. Sasaki-like almost contact B-metric manifolds belong to the class $\F_4$ of the Ganchev-Mihova-Gribachev classification \cite{GaMGri} and they are an analog to the class of Sasakian manifolds in the theory of almost contact metric manifolds. It is known that the $(1,1)$-tensor $\varphi $ of an almost contact metric manifold
(with a Riemannian or pseudo-Riemannian metric ) acts as an isometry with respect to the metric on the contact distribution. A distinguishing feature of an almost contact B-metric manifold  from an almost contact metric manifold is that this action of $\varphi $ is an anti-isometry. Also, the metric on an almost contact B-metric manifold is necessarily pseudo-Riemannian (known as a B-metric) and its restriction on the contact distribution is of neutral signature. Moreover, on a such manifold there exists a second B-metric, called an associated B-metric. In these reasons  the geometry of Sasaki-like almost contact B-metric manifolds is different from this  of Sasaki manifolds.\\
The paper is organized as follows.
Section 2 is a brief review of almost contact B-metric manifolds and geometry of null curves in a 3-dimensional Lorentzian manifold.
In Section 3  we investigate slant null curves  in a 3-dimensional Sasaki-like almost contact B-metric manifold. Since in general the slant  curves are not invariant under a reparameterization, we study such curves with respect to the original parameter. For the examined curves we find the general Frenet frame 
${\bf F}$ and the Frenet frame ${\bf\overline F}$ for which the original parameter is distinguished, as well as the curvatures with respect to ${\bf F}$ and ${\bf\overline F}$. The main result is Theorem \ref{Theorem 3.11} where we prove that both curvatures  $\overline k_1$ and $\overline k_2$ of a framed null slant curve $C$ are constants if and only if the function $b=g(\dot C,\varphi \dot C)$ ($b$ is a nonzero function in general in contrast to the almost contact metric and almost paracontact metric manifolds) is a constant on $C$. As a consequence from this theorem we obtain a necessary and sufficient condition a slant null curve to be a generalized helix. Section 4 is devoted to Legendre null curves. We determine for such curves the Frenet frames ${\bf F}$, ${\bf\overline F}$ and the corresponding  curvatures. We prove that there exist no geodesic Legendre null curves in a 3-dimensional Sasaki-like almost contact B-metric manifold. Also, we give a necessary and sufficient condition a Legendre null curve to be a null cubic. In Section 5 we consider some slant and Legendre null curves with respect to the one B-metric  in a 3-dimensional Sasaki-like almost contact B-metric manifold. We prove that with respect to 
the other B-metric these curves are slant and Legendre non-null curves, respectively, and characterize these curves geometrically. 
In the last Section 6 we give several examples illustrating relevant assertions in the previous sections. 
We construct slant null curves in a 3-dimensional Lie group which is a Sasaki-like almost contact B-metric manifold and find the matrix representation of considered curves.
\section{Preliminaries}\label{sec-2}
A $(2n+1)$-dimensional smooth manifold  $(M,\varphi,\xi ,\eta ,g)$ is called an almost contact manifold with B-metric  (or {\it an almost contact B-metric manifold}) \cite{GaMGri} if it is endowed with an almost contact structure $(\varphi ,\xi ,\eta )$ consisting of an endomorphism $\varphi $ of the tangent bundle, a Reeb vector field $\xi $  and its dual 1-form $\eta $, satisfying the following relations:
\begin{align*}
\varphi^2X=-X+\eta(X)\xi, \qquad \quad \eta(\xi )=1.
\end{align*}
Also, $M$ is equipped with a pseudo-Riemannian metric $g$, called {\it a B-metric} \cite{GaMGri}, determined by
\[
g(\varphi X,\varphi Y)=-g(X,Y)+\eta (X)\eta (Y).
\]
Here and further $X$, $Y$, $Z$ are tangent vector fields on $M$, \ie $X,Y,Z \in TM$.
Immediate consequences of the above conditions are:
\begin{align*}
\eta \circ \varphi =0, \quad \varphi \xi =0, \quad {\rm rank}(\varphi)=2n, \quad \eta (X)=
g(X,\xi ), \quad g(\xi,\xi )=1.
\end{align*}
The distribution $\mathbb {D}: x \in M \longrightarrow \mathbb {D}_x\subset T_xM$, where
\[
\mathbb D_x=Ker \eta=\{X_x\in T_xM: \eta (X_x)=0\}
\]
is called {\it a contact distribution} generated by 
$\eta $. Then the tangent space $T_xM$ at each $x\in M$ is the following orthogonal direct sum
\[
T_xM=\mathbb D_x\oplus span_\mathbb R\{\xi _x\} .
\]
The tensor field $\varphi $ induces an almost complex structure on each
fibre on $\mathbb D$. Since $g$ is non-degenerate metric on $M$ and $\xi $ is non-isotropic,
the contact distribution $\mathbb D$ is non-degenerate and the restriction $g_{\vert \mathbb D}$ of the metric $g$ on $\mathbb D$ is of signature $(n,n)$.
\\
The tensor field ${\widetilde g}$ of type $(0,2)$ given by 
${\widetilde g}(X,Y)=g(X,\varphi Y)+\eta (X)\eta (Y)$
is a B-metric, called {\it an associated metric} to $g$. Both metrics $g$ and 
${\widetilde g}$ are necessarily of signature $(n+1,n)$ 
$(+\ldots + -\ldots -)$.
\\
Let $\nabla$ be the Levi-Civita connection of $g$. 
The tensor field $F$ of type $(0,3)$ on $M$ is defined by
$F(X,Y,Z)=g((\nabla_X\varphi)Y,Z)$ 
and it has the following properties:
\[
F(X,Y,Z)=F(X,Z,Y)=F(X,\varphi Y,\varphi Z)+\eta (Y)F(X,\xi,Z)+\eta (Z)F(X,Y,\xi ).
\]
Moreover, we have
\begin{align}\label{2.1}
F(X,\varphi Y,\xi )=(\nabla _X\eta )Y=g(\nabla _X\xi,Y).
\end{align}
A classification of the almost contact B-metric manifolds with respect to $F$ is given in \cite{GaMGri} and
eleven basic classes  $\F_i$ $(i=1,2,\dots,11)$ are obtained. 
\\
The special class $\F_0$ is the intersection of all basic classes and it is determined by the condition $F=0$.
\\
The lowest possible dimension of the considered manifolds is three. The class of the 3-dimensional almost contact B-metric manifolds is
$\F_1\oplus \F_4\oplus \F_5\oplus \F_8\oplus \F_9\oplus \F_{10}\oplus \F_{11}$ \cite{GaMGri}. 
\\
In this paper we consider 3-dimensional Sasaki-like almost contact B-metric manifolds  (also known as almost contact complex Riemannian manifolds) which were introduced in \cite{IMM}. A Sasaki-like almost contact B-metric manifold is determined by the condition
\[
(\nabla _X\varphi )Y=-g(X,Y)\xi -\eta (Y)X+2\eta(X)\eta(Y)\xi ,
\]
or equivalently
\begin{align}\label{2.2}
\begin{array}{ll}
F(X,Y,Z)=g(\varphi X,\varphi Y)\eta (Z)+g(\varphi X,\varphi Z)\eta (Y). 
\end{array}
\end{align}
Any Sasaki-like almost contact B-metric manifold belongs to the class $\F_4$ \cite{GaMGri}.
Taking into account \eqref{2.1} and \eqref{2.2} for Sasaki-like manifold we have 
\begin{align}\label{2.3}
  \nabla _X\xi=-\varphi X .
\end{align}
\\
Let ${\widetilde \nabla }$ be the Levi-Civita connection of ${\widetilde g}$. We consider the symmetric tensor field $\Phi $ of type $(1,2)$ defined by
$\Phi (X,Y)={\widetilde \nabla }_XY-\nabla _XY$. For a 3-dimensional Sasaki-like  manifold the following equality holds \cite{MM}: 
\begin{align}\label{2.4}
{\widetilde \nabla }_XY-\nabla _XY=-\{g(X,\varphi Y)-g(\varphi X,\varphi Y)\}\xi .
\end{align}
Let us remark that a 3-dimensional almost contact B-metric manifold is a Lorentzian manifold. In the remainder of this section we briefly recall the main notions about null curves in a 3-dimensional Lorentzian manifold $M$ for which we refer to \cite{D-B, D-J}.

Let 
$C: I\longrightarrow M$ be a smooth curve in $M$ given locally by
\[
x_i=x_i(t), \quad t\in I\subseteq {\R}, \quad i\in \{1,2,3\}
\]
for a coordinate neighborhood $U$ of $C$. The tangent vector field is given by
\[
\frac{{\rm d}}{{\rm d}t}=(\dot {x}_1, \dot {x}_2, \dot {x}_3)=\dot {C},
\]
where we denote $\frac{{\rm d}x_i}{{\rm d}t}$ by $\dot {x}_i$ for $i\in \{1,2,3\}$.\\  
The smooth curve $C$ is said to be {\it a null (lightlike) curve} in $(M, g)$, if at each point $x$ of $C$ we have
\begin{align}\label{2'}
g(\dot {C},\dot {C})=0,\qquad \dot {C}\neq 0 \quad \text{for}\quad \forall t\in I.
\end{align}
%
A general Frenet frame on $M$ along $C$ is denoted by ${\bf F}=\{\dot {C}, N, W\}$ and the vector fields in ${\bf F}$ are determined by
\begin{align}\label{3'}
g(\dot {C},N)=g(W,W)=1, \quad g(N,N)=g(N,W)=g(\dot {C},W)=0.
\end{align}
In \cite[Theorem 1.1, p. 53]{D-B} it was proved that for a given $W$ there exists a unique $N$ satisfying
the corresponding equalities in \eqref{3'}.
The general Frenet equations with respect to ${\bf F}$ and $\nabla $ of $(M, g)$ are known from \cite{D-B, D-J} %
\begin{align}\label{2.7}
\nabla _{\dot {C}}\dot {C}=h\dot {C}+k_1W ,
\end{align}
\begin{align}\label{2.8}
\nabla _{\dot {C}}N=-hN+k_2W ,
\end{align}
\begin{align}\label{2.9}
\nabla _{\dot {C}}W=-k_2\dot {C}-k_1N,
\end{align}
where 
$h$, $k_1$ and  $k_2$
are smooth functions on $U$. The functions $k_1$ and  $k_2$ are called {\it curvature functions} of $C$.
\\
The general Frenet frame ${\bf F}$ and its general Frenet equations \eqref{2.7}, \eqref{2.8}, \eqref{2.9} are not unique as they depend on the parameter and the choice of the
screen vector bundle $S(TC^\bot )={\rm span}\{W\}$ of $C$ (for details see \cite[pp. 56-58]{D-B}, \cite[pp. 25-29]{D-J}). It is known \cite[p. 58]{D-B} that there exists a
parameter $p$ called a {\it distinguished parameter}, for which the function $h$ vanishes in \eqref{2.8}. The pair $(C(p), {\bf F})$, where ${\bf F}$ is a Frenet frame along $C$ with respect to a distinguished parameter $p$, is called a {\it framed null curve} (see \cite{D-J}). In general, 
$(C(p), {\bf F})$ is not unique since it depends on both $p$ and the screen distribution. A Frenet frame with the minimum number of curvature functions is called {\it Cartan Frenet frame} of a null curve $C$. In \cite{D-J} it is proved that if the null curve $C(p)$ is non-geodesic
such that  for $\ddot{C}=\frac{{\rm d}}{{\rm d}p}\dot{C}$ the condition $g(\ddot {C},\ddot {C})=k_1=1$ holds, then there exists only one Cartan Frenet frame
${\bf F}$ 
with the following Frenet equations
\begin{align}\label{2.10}
\begin{array}{lll}
\nabla _{\dot {C}}\dot {C}=W, \quad
\nabla _{\dot {C}}N=k_2W, \quad
\nabla _{\dot {C}}W=-k_2\dot {C}-N.
\end{array}
\end{align}
The latter equations are called the {\it Cartan Frenet equations} of $C(p)$ whereas $k_2$ is called a \emph{Cartan curvature or torsion function} and it is invariant upto a sign under Lorentzian transformations. A null curve together with its Cartan Frenet frame is called a {\it null Cartan curve}. Note that some authors \cite{H-I} term a framed null curve  $(C(p), {\bf F})$ a null Cartan curve and a Frenet frame 
${\bf F}$ along $C$ with respect to a distinguished parameter $p$ -- Cartan Frenet frame.  

\section{Slant null curves in 3-dimensional Sasaki-like almost contact B-metric manifolds}\label{sec-3}
\begin{defn}\label{Definition 3.1}
We say that a smooth curve $C(t)$ in an almost contact B-metric manifold $(M,\varphi,\xi,\eta,g)$ is {\it 
a slant curve} if 
\begin{align}\label{3.1}
g(\dot {C}(t),\xi )=\eta (\dot {C}(t))=a ,
\end{align}
where $a$ is a real constant.
 The curve $C$ is called a {\it Legendre curve} if $a=0$.
\end{defn}
Since for an almost contact B-metric manifold $M$ $g(X,\varphi X)$ is not zero in general, for the sake of brevity, we will use the following denotations
\[
b=g(\dot {C},\varphi \dot {C}), \quad \dot {C} (b)=\dot b ,
\]
where $b$ is a smooth function on a curve $C$ in $M$.
\begin{prop}\label{Proposition 3.2}
In a 3-dimensional almost contact B-metric manifold \\ $(M,\varphi,\xi,\eta,g)$ there exists no a slant null curve $C$ for which  $a=0$ and the function $b$ vanishes identically on 
$C$.
\end{prop}
\begin{proof} Let us assume that there exists a slant null curve $C$ in $M$ such that $a=0$ and $b$ vanishes identically on $C$. From $\eta (\dot {C})=\eta (\varphi \dot {C})=0$ it follows that $\dot {C}$ and $\varphi \dot {C}$ belong to the contact distribution $\mathbb {D}$ of $M$ along $C$. Since 
$\dot {C}$ and $\varphi \dot {C}$ are linearly independent, they form a basis of $\mathbb {D}$ at each point $x$ of $C$. Hence, for an arbitrary vector field $X \in \mathbb {D}_{\vert C}$ we have $X=u\dot {C}+v\varphi \dot {C}$ for some functions $u$ and $v$. By using $g(\dot {C},\dot {C})=g(\varphi \dot {C},\varphi \dot {C})=0$ and $b=0$ we obtain $g(X,X)=0$. The last implies a contradiction since $g_{\vert \mathbb D}$ is of signature $(1,1)$, which confirms our assertion.
\end{proof}
\begin{rem}\label{Remark 3.3}
Let $C(t)$ be a slant curve in an almost contact B-metric manifold $(M,\varphi,\xi,\eta,g)$. 
If we change the parameter $t$ of $C(t)$ with another parameter $p$, then we have $\dot C(p)=\dot C(t)\frac{{\rm d}t}{{\rm d}p}$. Hence \eqref{3.1} becomes
\begin{align*}
\begin{array}{ll}
g(\dot {C}(p),\xi )=\eta (\dot {C}(p))=\frac{{\rm d}t}{{\rm d}p}\eta (\dot {C}(t))=\frac{{\rm d}t}{{\rm d}p}a .
\end{array}
\end{align*}
Therefore $\eta (\dot {C}(p))$ is constant if and only if $t=\alpha p+\beta $, where $\alpha, \, \beta $ are constant. Since in general the slant  curves are not invariant under a reparameterization,
in the present paper we study slant null curves with respect to its original parameter.
\end{rem}
First, for latter use, we give the following  lemma
\begin{lem}\label{Lemma 3.4}
Let $C(t)$ be a slant null curve in a 3-dimensional Sasaki-like almost contact B-metric manifold. Then we have
\begin{align}\label{3.2}
g(\nabla _{\dot {C}}\dot {C},\xi )=-g(\dot {C},\nabla _{\dot {C}}\xi )=b ,
\end{align}
\begin{align}\label{3.3}
g(\nabla _{\dot {C}}\dot {C},\dot {C})=0 ,
\end{align}
\begin{align}\label{3.4}
g(\nabla _{\dot {C}}\dot {C},\varphi \dot {C})=\frac{\dot b-2a^3}{2}, \quad 
g(\nabla _{\dot {C}}\varphi \dot {C},\dot {C})=\frac{\dot b+2a^3}{2},
\end{align}
\begin{align}\label{3.5}
g(\nabla _{\dot {C}}\varphi \dot {C},\xi)=-g(\varphi \dot {C},\nabla _{\dot {C}}\xi )=a^2 .
\end{align}
\end{lem}
\begin{proof}
Since \eqref{3.1} is valid, we have $\dot C(\eta (\dot C))=0$. Then by using \eqref{2.3} we obtain \eqref{3.2}. The equality \eqref{3.3} is an immediate consequence from \eqref{2'}. From the expressions
\begin{align*}
\begin{array}{ll}
F(\dot {C},\dot {C},\dot {C})=g(\nabla _{\dot {C}}\dot {C},\varphi \dot {C})-g(\varphi (\nabla _{\dot {C}}\dot {C}),\dot {C}) ,\\ \\
\dot b=g(\nabla _{\dot {C}}\dot {C},\varphi \dot {C})+g(\dot {C},\nabla _{\dot {C}}\varphi \dot {C})
\end{array}
\end{align*}
we find $g(\nabla _{\dot {C}}\dot {C},\varphi \dot {C})=\displaystyle\frac{\dot b-F(\dot C,\dot C,\dot C)}{2}$ and
$g(\nabla _{\dot {C}}\varphi \dot {C},\dot {C})=\displaystyle\frac{\dot b+F(\dot C,\dot C,\dot C)}{2}$.
By virtue of \eqref{2.2} we get $F(\dot C,\dot C,\dot C)=2a^3$ which confirms the truthfulness of \eqref{3.4}. The equality $g(\varphi \dot C,\xi )=0$ and \eqref{2.3} imply \eqref{3.5}.
\end{proof}
Now, taking into account Proposition \ref{Proposition 3.2}, it is easy to see that the triad of vector fields
$\{\xi, \dot {C}, \varphi \dot {C}\}$ is a basis of $T_xM$ at each point $x$ of a slant null curve in a 3-dimensional almost contact B-metric manifold $(M,\varphi,\xi,\eta,g)$. By using this basis, we obtain
\begin{prop}\label{Proposition 3.5}
Let $C(t)$ be a slant null curve in a 3-dimensional Sasaki-like almost contact B-metric manifold. Then 
$\nabla _{\dot {C}}\dot {C}$ and the curvature $k_1$ are given by 
\begin{align}\label{3.6}
\nabla _{\dot {C}}\dot {C}=\alpha \xi +\beta \dot {C}+\gamma \varphi \dot {C},
\end{align}
where $\alpha $, $\beta $, $\gamma $ are the following functions along $C$
\begin{align}\label{3.7}
\alpha =\frac{-b(-2(a^4+b^2)+a\dot b)}{2(a^4+b^2)} ,
\end{align}
\begin{align}\label{3.8}
\beta =\frac{b\dot b}{2(a^4+b^2)} ,
\end{align}
\begin{align}\label{3.9}
\gamma =\frac{a(-2(a^4+b^2)+a\dot b)}{2(a^4+b^2)}
\end{align}
and
\begin{align}\label{3.10}
k_1=\epsilon\frac{-2(a^4+b^2)+a\dot b}{2\sqrt{a^4+b^2}}, \quad \epsilon =\pm 1.
\end{align}
\end{prop}
\begin{proof}
Let with respect to the basis $\{\xi, \dot {C}, \varphi \dot {C}\}$ of $T_xM$, $x\in C$,  $\nabla _{\dot {C}}\dot {C}$ have the decomposition \eqref{3.6}. Using \eqref{3.2}, \eqref{3.3} and the first equality in \eqref{3.4} we obtain the following linear system of equations for $\alpha $, $\beta $, $\gamma $
\begin{align*}
\left|\begin{array}{lll}
g(\nabla _{\dot {C}}\dot {C},\xi )=b=\alpha +a\beta \\ \\
g(\nabla _{\dot {C}}\dot {C},\dot {C})=0=a\alpha +b\gamma \\ \\
g(\nabla _{\dot {C}}\varphi \dot {C},\dot {C})=\displaystyle\frac{\dot b-2a^3}{2}=b\beta +
a^2\gamma  .
\end{array}\right.
\end{align*}
The determinant of the above system is $\Delta=-a^4-b^2$ and it is nonzero according to Proposition \ref{Proposition 3.2}. Hence, the system has a unique solution given by \eqref{3.7},  \eqref{3.8} and \eqref{3.9}. From \eqref{2.7} it follows that 
\begin{align}\label{3.11}
k_1^2=g(\nabla _{\dot {C}}\dot {C},\nabla _{\dot {C}}\dot {C})=\alpha ^2+a^2\gamma ^2.
\end{align}
Substituting \eqref{3.7} and \eqref{3.9} in the latter equality we obtain \eqref{3.10}.
\end{proof}
\begin{cor}\label{Corollary 3.6}
A slant null curve $C(t)$ in a 3-dimensional  Sasaki-like almost contact B-metric manifold is geodesic if and only if 
\begin{align}\label{3.12}
b=a^2\tan (2a(t+C_1)), \quad C_1=const.
\end{align}
\end{cor}
\begin{proof}
It is known \cite{D-B} that a null curve is geodesic if and only if $k_1$ vanishes. From \eqref{3.10}  it follows that $k_1=0$ if and only if $\dot b=\frac{2(a^4+b^2)}{a}$. Integrating  this equality implies \eqref{3.12}.
\end{proof}
Now, let us assume that $C(t)$ is a non-geodesic slant null curve in a 3-dimensional  Sasaki-like almost contact B-metric manifold. Since $k_1\neq 0$, from \eqref{3.11} it follows that the vector field $\nabla _{\dot {C}}\dot {C}=\alpha \xi +\beta \dot {C}+\gamma \varphi \dot {C}$ is spacelike. It is not hard to see that the unique screen vector bundles of $C(t)$ are $S(TC^\bot)={\rm span}\{W\}$ and
$\bar{S}(TC^\bot)={\rm span}\{\overline W\}$, where $W$ and $\overline W$ are the following unit spacelike vector fields orthogonal to $\dot C$
\begin{align}\label{3.13}
W=\frac{1}{k_1}(\alpha \xi +\gamma \varphi \dot {C})=\frac{-\epsilon b}{\sqrt{a^4+b^2}}\xi +\frac{\epsilon a}{\sqrt{a^4+b^2}}\varphi \dot {C},
\end{align}
\begin{align}\label{3.14}
\begin{array}{ll}
\overline W=\displaystyle\frac{\epsilon _1}{k_1}(\alpha \xi +\beta \dot {C}+\gamma \varphi \dot {C})=
\displaystyle\epsilon _1\left(\frac{-\epsilon b}{\sqrt{a^4+b^2}}\xi \right.\\ \\
\qquad \left.+\displaystyle\frac{\epsilon b\dot b}{\sqrt{a^4+b^2}(-2(a^4+b^2)+a\dot b)}\dot C+\frac{\epsilon a}{\sqrt{a^4+b^2}}\varphi \dot {C}\right),
\end{array}
\end{align}
where $\epsilon _1=\pm 1$. From \eqref{3.13} and \eqref{3.14} we have
\begin{align}\label{3.15}
\overline W=\epsilon _1\left(W+\frac{\beta }{k_1}\dot C\right).
\end{align}

We denote by ${\bf F}=\{\dot {C}, N, W\}$ and 
${\bf \overline F}=\{\dot {C}, \overline N, \overline W\}$ the Frenet frames with respect to $\{t, S(TC^\bot)\}$ and $\{t, \bar S(TC^\bot)\}$, respectively. Let $h$, $k_1$, $k_2$ and $\bar h$, $\bar k_1$, $\bar k_2$ be the curvature functions of 
$C$ with respect to ${\bf F}$ and ${\bf \overline F}$, respectively.
Taking into account \eqref{3.13}, \eqref{3.14}, the equality \eqref{3.6} becomes
\begin{align}\label{3.16}
\nabla _{\dot {C}}\dot {C}=\beta \dot {C}+k_1W,
\end{align}
\begin{align}\label{3.17}
\nabla _{\dot {C}}\dot {C}=\bar k_1\overline W,
\end{align}
respectively, where 
\begin{align}\label{3.18}
\bar k_1=\epsilon _1k_1 .
\end{align}
By using  \eqref{3.15} and the relations between the elements of ${\bf F}$ and 
${\bf \overline F}$ from \cite[p. 65]{D-B} we obtain
\begin{align}\label{3.19}
\overline N=-\frac{\beta^2}{2k_1^2}\dot C+N-\frac{\beta}{k_1}W ,
\end{align}
\begin{align}\label{3.20}
\bar k_2=\epsilon _1
\left(k_2-\frac{{\rm d}}{{\rm d}t}{\left(\frac{\beta}{k_1}\right)}-\frac{\beta ^2}{2k_1}\right).
\end{align}
\begin{prop}\label{Proposition 3.7}
Let $M$ be a 3-dimensional Sasaki-like almost contact B-metric manifold and $C(t)$ be a non-geodesic slant null curve in $M$.
\par
\item (i) If the screen vector bundle of $C(t)$ is $S(TC^\bot)={\rm span}\{W\}$, then
\begin{align}\label{3.21}
N=\frac{a^3}{a^4+b^2}\xi -\frac{a^2}{2(a^4+b^2)}\dot C+\frac{b}{a^4+b^2}\varphi \dot C ,
\end{align}
\begin{align}\label{3.22}
h=\beta =\frac{b\dot b}{2(a^4+b^2)} ,
\end{align}
\begin{align}\label{3.23}
k_2=\epsilon \frac{-2a^2(a^4+b^2)+3a^3\dot b}{4(a^4+b^2)^{3/2}} , \quad \epsilon =\pm 1.
\end{align}
\par
\item (ii) If the screen vector bundle of $C(t)$ is $\bar S(TC^\bot)={\rm span}\{\overline W\}$, then the original parameter $t$ of $C(t)$ is distinguished and
\begin{align}\label{3.24}
\begin{array}{ll}
\overline N=\displaystyle\left(\frac{\dot b-2a^3}{-2(a^4+b^2)+a\dot b}\right)\xi -
\left(\frac{4a^2(a^4+b^2-a\dot b)+\dot b^2}{2(-2(a^4+b^2)+a\dot b)^2}\right)\dot C\\ \\
\quad\, -\displaystyle\left(\frac{2b}{-2(a^4+b^2)+a\dot b}\right)\varphi \dot C ,
\end{array}
\end{align}
\begin{align}\label{3.25}
\bar k_2=\epsilon_1\epsilon\frac{(a^4+b^2)(8b\ddot {b}+20a^3\dot b-8a^6-8a^2b^2)-a\dot b^3-2(3a^4+7b^2)\dot b^2}{4\sqrt{a^4+b^2}(-2(a^4+b^2)+a\dot b)^2}.
\end{align}
\end{prop}
\begin{proof}
(i) Let $N=\lambda \xi +\mu \dot {C}+\nu \varphi \dot {C}$ for some functions $\lambda$, $\mu$ and 
$\nu$ along $C$. By using \eqref{3'} and \eqref{3.13} we obtain the following system for $\lambda$, 
$\mu$, $\nu$
\begin{align}\label{3.26}
\left|\begin{array}{lll}
g(N,N)=0=\lambda^2+a^2\nu^2+2a\lambda \mu+2b\mu \nu \\
g(N,\dot C)=1=a\lambda +b\nu \\
g(N,W)=0=\displaystyle\frac{-\epsilon b}{\sqrt{a^4+b^2}}\lambda +\frac{\epsilon a^3}{\sqrt{a^4+b^2}}\nu .
\end{array}\right.
\end{align}
The last two equations in \eqref{3.26} form the linear system 
\begin{align}\label{3.27}
\left|\begin{array}{ll}
a\lambda +b\nu =1\\
-b\lambda +a^3\nu =0 ,
\end{array}\right.
\end{align}
whose determinant $\Delta _1=a^4+b^2$ is nonzero. Thus \eqref{3.27} has a unique solution that is
\begin{align}\label{3.28}
\lambda =\frac{a^3}{a^4+b^2} , \qquad \nu =\frac{b}{a^4+b^2} .
\end{align}
Substituting \eqref{3.28} in the first equation of \eqref{3.26} we find 
\begin{align}\label{3.29}
\mu =-\frac{a^2}{2(a^4+b^2)} .
\end{align}
Hence, for the vector field $N$ from the Frenet frame ${\bf F}$ the equality \eqref{3.21} is valid.
The equality \eqref{3.22} for the curvature function $h$ with respect to ${\bf F}$ follows from \eqref{3.8} and \eqref{3.16}.\\
From the Frenet equation \eqref{2.8} with respect to ${\bf F}$  we derive 
\[
k_2=g(\nabla _{\dot {C}}N,W).
\]
Taking into account \eqref{3.13}, the latter equality becomes
\begin{align*}
\begin{array}{ll}
k_2=\displaystyle\frac{-\epsilon b}{\sqrt{a^4+b^2}}\left(\dot \lambda +\lambda g(\nabla _{\dot {C}}\xi ,\xi )+\dot \mu a+\mu g(\nabla _{\dot {C}}\dot {C},\xi )+
\nu g(\nabla _{\dot {C}}\varphi \dot {C},\xi )\right)\\
+\displaystyle\frac{\epsilon a}{\sqrt{a^4+b^2}}\left(\lambda g(\nabla _{\dot {C}}\xi ,\varphi \dot {C})+\dot \mu b+\mu g(\nabla _{\dot {C}}\dot {C},\varphi \dot {C})+\dot \nu a^2+
\nu g(\nabla _{\dot {C}}\varphi \dot {C},\varphi \dot {C})\right),
\end{array}
\end{align*}
where $\lambda$, $\nu$ and $\mu$  are given by \eqref{3.28} and \eqref{3.29}, respectively. After direct
computations using Lemma \ref{Lemma 3.4} we obtain \eqref{3.23}.\\
(ii) From \eqref{3.17} it follows that $\overline h=0$  which means that the original parameter $t$ of $C$ is distinguished.
With the help of \eqref{3.8}, \eqref{3.10}, \eqref{3.13}, \eqref{3.19}, \eqref{3.21} we obtain \eqref{3.24}. Using \eqref{3.8}, \eqref{3.10}, \eqref{3.20}, \eqref{3.23} we get \eqref{3.25}.
\end{proof}
According to the assertion (ii) in Proposition \ref{Proposition 3.7} a non-geodesic slant null curve $C(t)$ in a 3-dimensional Sasaki-like almost contact B-metric manifold is a framed null curve 
with respect to a Frenet frame ${\bf \overline F}=\{\dot {C}, \overline N, \overline W\}$ and with a distinguished parameter--the original parameter $t$. Such curve we will denote by 
$(C(t),{\bf \overline F})$.
At the end of this section we study null Gartan helices and generalized helices among the framed null curves $(C(t),{\bf \overline F})$.
Let us recall 
\begin{defn}\label{Definition 3.8}\cite{D-J, FGL}
A null Cartan curve with a constant Cartan curvature $k_2$ is called a null Cartan helix.
\end{defn}
\begin{defn}\label{Definition 3.9}\cite{D-J, FGL}
A null Cartan curve $C$ in a 3-dimensional Lorentzian manifold $(M,g)$ is a generalized helix if there exists a nonzero constant vector $v$ such that $g(\dot C,v)$ is constant.
\end{defn}
The following result is known from \cite{D-J, FGL}
\begin{prop}\label{Proposition 3.10}\cite{D-J, FGL}
Let  $C$ be a null Cartan curve in a 3-dimensional Lorentzian manifold. Then $C$ is a generalized helix  if and only if $C$ is a null Cartan helix.
\end{prop}
\begin{thm}\label{Theorem 3.11}
Let $C(t): I\longrightarrow M$ be a non-geodesic slant null curve in a 3-dimensional Sasaki-like almost contact B-metric manifold $M$ and the function $b(t)\in C^1$ in $I$. Then the framed null curve 
$(C(t),{\bf \overline F})$ is of constant curvatures $\bar k_1$ and $\bar k_2$ if and only if $b$ is a constant in $I$. Moreover, for an arbitrary constant $b$ the curvatures $\bar k_1$ and $\bar k_2$ are the following constants
\begin{align}\label{3.35}
\bar k_1=\epsilon _1\epsilon (-\sqrt{a^4+b^2}),
\end{align}
\begin{align}\label{3.36}
\bar k_2=\epsilon _1\epsilon \left(-\frac{a^2}{2\sqrt{a^4+b^2}}\right) .
\end{align}
\end{thm}
\begin{proof}
''$\Longrightarrow $'' Let us assume that both $\bar k_1$ and $\bar k_2$ are constants. Taking into account \eqref{3.10} and \eqref{3.18} we have $\bar k_1=\epsilon_1k_1=\epsilon_1\epsilon C_1$, where 
\begin{align}\label{3.30}
C_1=\displaystyle\frac{-2(a^4+b^2)+a\dot b}{2\sqrt{a^4+b^2}}
\end{align}
is a nonzero constant. From the above equality we obtain
\begin{align}\label{3.31}
\dot b=\frac{2C_1\sqrt{a^4+b^2}+2(a^4+b^2)}{a} .
\end{align}
By using \eqref{3.31} we find
\begin{align}\label{3.32}
\ddot b=\frac{2b\dot b(C_1+2\sqrt{a^4+b^2})}{a\sqrt{a^4+b^2}} .
\end{align}
Now, substituting $\bar k_2=C_2$=const, $\bar k_1=\epsilon _1\epsilon C_1$ and \eqref{3.23} in \eqref{3.20}, we get 
\[
\begin{array}{ll}
(3b^2-4a^4)\dot b^2-4b(a^4+b^2)\ddot b+2C_1\sqrt{a^4+b^2}(-2a^2(a^4+b^2)+3a^3\dot b)\\
-8\epsilon _1\epsilon C_1C_2(a^4+b^2)^2=0 .
\end{array}
\]
With the help of \eqref{3.31} and \eqref{3.32} the latter equality becomes
\begin{align}\label{3.33}
5b^2+6C_1\sqrt{a^4+b^2}+2\epsilon _1\epsilon C_1C_2+4a^4+C_1^2=0 .
\end{align}
After differentiation of  \eqref{3.33} we obtain
\begin{align}\label{3.34}
\left(b^2\right)^{.}\left(5\sqrt{a^4+b^2}+3C_1\right)=0,
\end{align}
where $\left(b^2\right)^{.}$ stands for $\frac{{\rm d}}{{\rm d}t}b^2$.
Further, we consider separately the cases $C_1>0$ and $C_1<0$ in \eqref{3.34}.
\\ 1) If $C_1>0$, then $5\sqrt{a^4+b^2}+3C_1$ is nonzero for any $t\in I$. Thus from \eqref{3.34}
we infer that $\left(b^2\right)^{.}=0$ for any $t\in I$, which implies that $b$ is a constant.
\\ 2) If $C_1<0$, then there exist the following possibilities for the vanishing of the left hand side of \eqref{3.34}:
\\ 2.1) $\left(b^2\right)^{.}=0$ for any $t\in I$ and hence $b$ is a constant.
\\ 2.2) There exists  $t_0\in I$ such that $\left(b^2\right)^{.}(t_0)\neq 0$. Since $\left(b^2\right)^{.}$ is continuous in $I$, it follows that there exists an open interval $I_0\subset I$, $t_0\in I_0$, such that  ${\left(b^2\right)^{.}(t)\neq 0}$ for any $t\in I_0$. Then for any $t\in I_0$ we have $\left(5\sqrt{a^4+b^2}+3C_1\right)=0$. From the latter equality it follows that $b^2$ is a constant in $I_0$ and hence $\left(b^2\right)^{.}=0$ in
$I_0$ which is a contradiction. 
\\ 2.3) Analogously as in 2.2) we see that the assumption $\left(b^2\right)^{.}\neq 0$ for any $t\in I$ implies a contradiction.
\\ ''$\Longleftarrow $'' Conversely, let $b$ be a constant in $I$. Substituting $\dot b=\ddot b=0$ in
\eqref{3.10} and \eqref{3.25} we obtain that $\bar k_1$ and $\bar k_2$ are constants given by \eqref{3.35} and \eqref{3.36}, respectively.
\end{proof}
\begin{cor}\label{Corollary 3.12}
A non-geodesic slant null curve $(C(t),{\bf \overline F})$ in a 3-dimensional Sasaki-like almost contact B-metric manifold is a generalized helix if and only if the following conditions are valid
\begin{align}\label{3.37}
\begin{array}{ll}
b^2=1-a^4, \quad \text{where}\quad a\in [-1,0)\cup (0,1];\\
\bar k_2=\frac{a^2}{2} .
\end{array}
\end{align}
\end{cor}
\begin{proof}
Taking into account Proposition \ref{Proposition 3.10} we see that
it is enough to show that $(C(t),{\bf \overline F})$ is a null Cartan helix if and only if \eqref{3.37} hold. 
\\
First, let $(C(t),{\bf \overline F})$ be a null Cartan helix. Hence $\bar k_1=1$ and $\bar k_2$ is a constant. Now, applying Theorem \ref{Theorem 3.11} we obtain that $b$ is a constant. Then  $\bar k_1$ and $\bar k_2$ are given by \eqref{3.35} and \eqref{3.36}, respectively. Replacing $\bar k_1$ in \eqref{3.35} with 1 we get $1=\epsilon _1\epsilon (-\sqrt{a^4+b^2})$. By using the latter equality and \eqref{3.36} we obtain the conditions \eqref{3.37}.
\\
Conversely, let us assume that \eqref{3.37} are fulfilled. Since $b$ is a constant, from Theorem \ref{Theorem 3.11} it follows that for $\bar k_1$ and $\bar k_2$ the equalities \eqref{3.35} and \eqref{3.36} hold. By using \eqref{3.36} and $\bar k_2=\frac{a^2}{2}$ we get $\epsilon _1\epsilon =-1$. Substituting the latter and $b^2=1-a^4$ in \eqref{3.35} we find $\bar k_1=1$ which completes the proof.
\end{proof}
As a generalization of the magnetic curves in \cite{Bejan} was introduced the notion of $F$-geodesics in a manifold $M$ endowed with a (1,1)-tensor field $F$ and with a linear connection $\nabla $.
\begin{defn}\label{Definition 3.13}\cite{Bejan}
A smooth curve $\gamma : I\longrightarrow M$ in a manifold $(M,F,\nabla )$ is an $F$-geodesic if 
$\gamma (t)$ satisfies $\nabla _{\dot {\gamma }(t)}\dot {\gamma }(t)=F\dot {\gamma }(t)$.
\end{defn}
Note that an $F$-geodesic is not invariant under a reparameterization.
\begin{prop}\label{Proposition 3.14}
Let $C(t)$ be a slant null curve in a 3-dimensional Sasaki-like almost contact B-metric manifold $(M,\varphi,\xi,\eta,g)$. Then 
\par
\item (i) $C(t)$ is a $\varphi $-geodesic if and only if $b=0$ and $a=-1$. 
\par
\item (ii) If $C(t)$ is a $\varphi $-geodesic, then it is a null Cartan curve with a Cartan Frenet frame 
${\bf \overline F}=\{\dot C, \, \overline N=-\xi -\frac{1}{2}\dot C, \, \overline W=\varphi \dot C\}$ and 
$\bar k_2=\frac{1}{2}$. 
\end{prop}
\begin{proof}
(i) $C(t)$ is a $\varphi $-geodesic if and only if 
\begin{equation}\label{3.38}
\nabla _{\dot C}\dot C=\varphi \dot C. 
\end{equation}
Using \eqref{3.6}, \eqref{3.7}, \eqref{3.8} and \eqref{3.9} we see that \eqref{3.38} is fulfilled if and only if $b=0$ and $a=-1$. 
\par
(ii) If $C(t)$ is a $\varphi $-geodesic, then from \eqref{3.38} and $g(\varphi \dot C,\varphi \dot C)=a^2=1$ it follows that $\overline W=\varphi \dot C$, $\bar k_1=1$. Thus,  $C(t)$ is a null Cartan curve. Now, replacing  $a$ and $b$ in \eqref{3.10}, \eqref{3.24}, \eqref{3.25} by $-1$ and $0$, respectively, we find $k_1=\epsilon (-1)$, $\overline N=-\xi -\frac{1}{2}\dot C$, $\bar k_2=\epsilon _1\epsilon \left(-\frac{1}{2}\right)$. By using  \eqref{3.18} and $\bar k_1=1$ we get $\epsilon _1\epsilon =-1$ which completes the proof.
\end{proof}
From Corollary \ref{Corollary 3.12} and Proposition \ref{Proposition 3.14} we obtain
\begin{cor}\label{Corollary 3.15}
A $\varphi $-geodesic slant null curve in a 3-dimensional Sasaki-like almost contact B-metric manifold 
is a generalized helix.
\end{cor}
\section{Legendre null curves in 3-dimensional Sasaki-like almost contact B-metric manifolds}\label{sec-4}
Let us emphasize that  according to Proposition \ref{Proposition 3.2} if $C$ is a Legendre  null curve in a 3-dimensional almost contact B-metric manifold, then  the function b is not identically zero.\\
Following the proof of Proposition \ref{Proposition 3.5} and Proposition \ref{Proposition 3.7} we infer that they are also valid if one of $a$ or $b$ is zero. Thus, substituting $a$ in \eqref{3.7}, \eqref{3.8}, \eqref{3.9} and \eqref{3.10} with $0$ we state
\begin{cor}\label{Corollary 3.16}
Let $C(t)$ be a Legendre null curve in a 3-dimensional Sasaki-like almost contact B-metric manifold. Then 
$\nabla _{\dot {C}}\dot {C}$ and the curvature $k_1$ are given by
\begin{equation}\label{3.39}
\nabla _{\dot {C}}\dot {C}=b\xi +\frac{\dot b}{2b}\dot C ,
\end{equation}
\begin{equation}\label{3.40}
k_1=\epsilon b, \quad \epsilon =\pm 1 .
\end{equation}
\end{cor}
The following proposition is an immediate consequence from \eqref{3.40} 
\begin{prop}\label{Proposition 3.17}
There exist no geodesic Legendre null curves in a 3-dimensional Sasaki-like almost contact B-metric manifold.
\end{prop}
We can rewrite \eqref{3.39} in the following two ways
\[
\nabla _{\dot {C}}\dot {C}=\frac{\dot b}{2b}\dot C+k_1W \quad \text{and}\quad \nabla _{\dot {C}}\dot {C}=\bar k_1\overline W
\]
where
\begin{equation}\label{3.41}
W=\epsilon \xi , \quad \epsilon =\pm 1, \quad \bar k_1=\epsilon _1\epsilon b, \quad\overline W=\epsilon _1\left(\epsilon \xi +\epsilon \frac{\dot b}{2b^2}\dot C\right), \, \,  \epsilon _1=\pm 1.
\end{equation}
Hence, the unique screen vector bundles of a Legendre null curve $C(t)$ are $S(TC^\bot)={\rm span}\{W\}$ and $\bar{S}(TC^\bot)={\rm span}\{\overline W\}$, where $W$ and $\overline W$ are the unit spacelike vector fields orthogonal to $\dot C$ given in \eqref{3.41}. Further, substituting $a$ in \eqref{3.21}, \eqref{3.22}, \eqref{3.23}, \eqref{3.24} and \eqref{3.25} with $0$ we obtain  the following corollary from Proposition \ref{Proposition 3.7}
\begin{cor}\label{Corollary 3.18}
Let $M$ be a 3-dimensional Sasaki-like almost contact B-metric manifold and $C(t)$ be a Legendre null curve in $M$.
\par
\item (i) If the screen vector bundle of $C(t)$ is $S(TC^\bot)={\rm span}\{W=\epsilon \xi \}$, then
\[
N=\frac{1}{b}\varphi \dot C, \qquad h=\frac{\dot b}{2b}, \qquad k_2=0.
\]
\item (ii) If the screen vector bundle of $C(t)$ is $\bar S(TC^\bot)={\rm span}\{\overline W\}$, where 
$\overline W$ is given in \eqref{3.41}, then the original parameter $t$ of $C(t)$ is distinguished and
\begin{equation}\label{3.42}
\overline N=-\frac{\dot b}{2b^2}\xi -\frac{\dot b^2}{8b^4}\dot C+\frac{1}{b}\varphi \dot C, \qquad \bar k_2=\epsilon _1\left(\frac{4b\ddot b-7\dot b^2}{8b^3}\right).
\end{equation}
\end{cor}
A Legendre null curve $(C(t),{\bf \overline F})$ is a framed null curve with respect to a Frenet frame ${\bf \overline F}=\{\dot {C}, \overline N, \overline W\}$, where $\overline W$ and $\overline N$ are given in
\eqref{3.41} and \eqref{3.42}.
\begin{defn}\label{Definition 3.19}\cite{D-J, FGL}
A null Cartan helix of zero Cartan curvature $k_2$ is called null cubic.
\end{defn}
Now, taking into account the expressions for $\bar k_1$ and $\bar k_2$ in \eqref{3.41} and \eqref{3.42}, respectively, we state
\begin{prop}\label{Proposition 3.20}
Let $(C(t),{\bf \overline F})$ be a Legendre null curve in a 3-dimensional Sasaki-like almost contact B-metric manifold $M$. Then 
\par
\item (i) $(C(t),{\bf \overline F})$ is of constant curvatures $\bar k_1$ and $\bar k_2$ if and only if $b$ is a nonzero constant.
\par
\item (ii) For a nonzero constant $b$ we have $\bar k_2=0$.
\par
\item (iii) $(C(t),{\bf \overline F})$ is a null cubic if and only if $\epsilon _1\epsilon b=1$.
\end{prop}

\section{Non-null slant curves in a 3-dimensional Sasaki-like almost
 contact  B-metric manifold induced from some slant and Legendre null curves}\label{sec-5}
A curve $\gamma : I\longrightarrow M$ in a 3-dimensional Lorentzian manifold $(M,g)$ is said to be {\it a unit speed curve} (or $\gamma $ is parameterized by arc length $s$) if $g(\gamma ^\prime ,\gamma ^\prime)=\epsilon _1=\pm 1$, where $\gamma ^\prime =\frac{{\rm d}\gamma }{{\rm d}s}$ is the velocity vector field. A unit speed curve 
$\gamma $ is said to be {\it spacelike} or {\it timelike} if $\epsilon _1=1$ or $\epsilon _1=-1$, respectively. A unit speed curve $\gamma $ is said to be {\it a Frenet curve} if 
one of the following three cases holds \cite{W}:
\begin{itemize}
\item $\gamma $ is of osculating order 1 that is $\nabla _{\gamma ^\prime }\gamma ^\prime =0$, i.e. 
$\gamma $ is a geodesic;
\item $\gamma $ is of osculating order 2, i.e. there exist two orthonormal vector fields $E_1$, $E_2$ and a positive function $k$ (the curvature) along $\gamma $ such that $E_1=\gamma ^\prime $, $g(E_2,E_2)=\epsilon _2=\pm 1$ and
\[
\nabla _{\gamma ^\prime }E_1=\epsilon _2kE_2 , \quad \nabla _{\gamma ^\prime }E_2=-\epsilon _1kE_1; 
\]
\item $\gamma $ is of osculating order 3, i.e. there exist three orthonormal vector fields $E_1$, $E_2$, 
$E_3$ and two positive functions $k$ (the curvature) and $\tau $ (the torsion) along $\gamma $ such that 
$E_1=\gamma ^\prime $, $g(E_2,E_2)=\epsilon _2=\pm 1$, $g(E_3,E_3)=\epsilon _3=\pm 1$, 
$\epsilon _3=-\epsilon _1\epsilon _2$ and 
\[
\nabla _{\gamma ^\prime }E_1=\epsilon _2kE_2 , \quad \nabla _{\gamma ^\prime }E_2=-\epsilon _1kE_1+\epsilon _3\tau E_3 , \quad \nabla _{\gamma ^\prime }E_3=-\epsilon _2\tau E_2. 
\]
\end{itemize} 
As in the case of Riemannian geometry, a Frenet curve  in a 3-dimensional Lorentzian manifold is a geodesic if and only if its curvature $k$ vanishes. Also a Frenet curve $\gamma $ with a curvature $k$ and a torsion $\tau $ is called \cite{I}:
\begin{itemize}
\item {\it a pseudo-circle} if $k=const$ and $\tau =0$;
\item {\it a helix} if $k=const$ and $\tau =const$;
\item {\it a proper helix} if $\gamma $ is a helix which is not a circle;
\item {\it a  generalized helix} if $\displaystyle\frac{k}{\tau}=const$ but $k$ and $\tau $ are not constant.
\end{itemize}
Analogously as in \cite{W} we say that a Frenet curve $\gamma (s)$ in an almost contact B-metric manifold $(M,\varphi,\xi,\eta,g)$ is  slant if $\eta (\gamma ^\prime (s))=a=const$.\\
Since there exist two B-metrics $g$ and $\widetilde g$ on an almost contact B-metric manifold $M$, we can consider a curve $\gamma $ in $M$ with respect to both metrics. \\
In this section we investigate non-null curves with respect to  $\widetilde g$ in a 3-dimensional Sasaki-like almost contact B-metric manifold, induced from Legendre and slant null curves with respect to $g$ with constant curvatures.
\begin{thm}\label{Theorem 4.1}
Let $M$ be a 3-dimensional Sasaki-like almost contact B-metric manifold and $(C(t), {\bf \overline F})$ be a Legendre null curve with respect to $g$ in $M$ of constant curvatures $\bar k_1$ and $\bar k_2$.
Then the curve $C$ with respect to $\widetilde g$ is
\par
\item (i) spacelike if $b>0$ or timelike if $b<0$;
\par
(ii) a Legendre curve;
\par
(iii) a geodesic.
\end{thm}
\begin{proof}
(i) From the assertion (i) in Proposition \ref{Proposition 3.20} it follows that $b$ is a nonzero constant. Thus $\widetilde g(\dot C,\dot C)=b\neq 0$. Now, we parameterize $C(t)$ by its 
arc length parameter $\widetilde s$ with respect to $\widetilde g$ given by
\[
\widetilde s=\int^{t}_{0}\sqrt{\vert \widetilde g(\dot C,\dot C)\vert }{\rm d}u=\int^{t}_{0}\sqrt{\vert b\vert }{\rm d}u =\sqrt{\vert b\vert }t .
\]
Then for the tangent vector $C^\prime (\widetilde s)=\displaystyle\dot C(t)\frac{{\rm d}t}{{\rm d}\widetilde s}=\frac{\dot C(t)}{\sqrt{\vert b\vert }}$ of the curve $C(\widetilde s)$ we have 
$\widetilde g(C^\prime ,C^\prime )=\displaystyle\frac{b}{\vert b\vert }=\pm 1$ which confirms the assertion (i).
\par
(ii) By direct calculations we find
\begin{align*}
\begin{array}{ll}
\widetilde \eta (C^\prime )=\widetilde g(C^\prime ,\xi )=\displaystyle\eta (C^\prime )=\frac{1}{\sqrt{\vert b\vert }}\eta (\dot C)=0 , 
\end{array}
\end{align*}
which shows that $C(\widetilde s)$ is a Legendre curve.\\
(iii) By virtue of \eqref{2.4} we get
\begin{align}\label{4.2}
\widetilde \nabla _{C^\prime }C^\prime =\frac{1}{\vert b\vert }\widetilde \nabla _{\dot C}\dot C=
\frac{1}{\vert b\vert }\left(\nabla _{\dot C}\dot C-b\xi \right).
\end{align}
Substituting $\dot b=0$ in \eqref{3.39} we obtain $\nabla _{\dot C}\dot C=b\xi $. The latter equality and \eqref{4.2} imply $\widetilde \nabla _{C^\prime }C^\prime =0$, i.e. $C(\widetilde s)$ is a geodesic.
\end{proof}
\begin{thm}\label{Theorem 4.2}
Let $(C(t), {\bf \overline F})$ be a slant null curve with respect to $g$ in a 3-dimensional 
Sasaki-like almost contact B-metric manifold $M$ for which $b=0$. Then $C$ with respect to 
$\widetilde g$ is a spacelike slant curve in $M$ such that:
\par
\item (i) $C$ is of osculating order 3. The orthonormal vector fields $E_1$, $E_2$, $E_3$ with respect to $\widetilde g$ are given as follows:
\begin{align}\label{4.3}
E_1(\widetilde s)=C^\prime (\widetilde s)=\frac{\dot C}{\vert a \vert}, \quad \widetilde g(E_1,E_1)=1 ,
\end{align}
where $\widetilde s$ is the arc length parameter of $C(t)$ with respect to $\widetilde g$,
\begin{align}\label{4.4}
E_2(\widetilde s)=-\frac{1}{a} \varphi \dot C+\xi , \quad
\, \widetilde g(E_2,E_2)=1,
\end{align}
\begin{align}\label{4.5}
E_3(\widetilde s)=\frac{1}{\vert a\vert}(\varphi \dot C-\dot C), \quad \widetilde g(E_3,E_3)=-1;
\end{align}
(ii) $C(\widetilde s)$ is a proper helix whose curvature $\widetilde k$ and torsion $\widetilde \tau $ are
\begin{align}\label{4.6}
\widetilde k(\widetilde s)=\widetilde \tau (\widetilde s)=1.
\end{align}
\end{thm}
\begin{proof}
Since $b=0$ for $(C(t), {\bf \overline F})$, from Proposition \ref{Proposition 3.2} it follows that $a\neq 0$.
First, for further use we compute:
\begin{align}\label{4.7}
\begin{array}{lll}
\widetilde g(\dot C,\dot C)=g(\dot C,\varphi \dot C)+(\eta (\dot C))^2=a^2 ,\, \, \widetilde g(\dot C,\varphi \dot C)=g(\varphi \dot C,\varphi \dot C)=a^2,\\
\widetilde g(\varphi \dot C,\varphi \dot C)=-\widetilde g(\dot C,\dot C)+(\eta (\dot C))^2=0. 
\end{array}
\end{align}
The curvature $\bar k_1(t)=\epsilon _1k_1$ and the vector fields $\overline W$, $\overline N$ from the frame ${\bf \overline F}$ along $C(t)$ we obtain by substituting $b=0$ in \eqref{3.10} and \eqref{3.14},  \eqref{3.24}, respectively. Thus we  
\begin{align}\label{4.8}
\bar k_1=\epsilon _1\epsilon (-a^2) ,
\end{align}
\begin{align}\label{4.9}
\overline W=\epsilon _1\epsilon \frac{1}{a}\varphi \dot C ,
\end{align}
\begin{align}\label{4.10}
\overline N=\frac{1}{a}\xi -\frac{1}{2a^2}\dot C .
\end{align}
Since $\widetilde g(\dot C,\dot C)=a^2\neq 0$, analogously as in the proof of Theorem \ref{Theorem 4.1} we parameterize $C(t)$ with respect to its arc length parameter $\widetilde s=\vert a\vert t$. Then it is easy to see that for the vector field $E_1(\widetilde s)=C^\prime (\widetilde s)$ the equality $\widetilde g(E_1,E_1)=1$ holds. Hence $C(\widetilde s)$ is a spacelike curve with respect to $\widetilde g$. Also, we have
\begin{align*}
\widetilde \eta (C^\prime )=\widetilde g(C^\prime ,\xi )=\displaystyle\eta (C^\prime )=\frac{1}{\vert a\vert }\eta (\dot C)=\frac{a}{\vert a\vert }=\pm 1 .
\end{align*}
From the above equality it is clear that the spacelike curve $C(\widetilde s)$ is a slant (non-Legendre) curve in $M$.\\
(i) By virtue of \eqref{2.4} we find
\begin{align}\label{4.11}
\widetilde \nabla _{C^\prime }C^\prime =\frac{1}{a^2}\widetilde \nabla _{\dot C}\dot C=
\frac{1}{a^2}\left(\nabla _{\dot C}\dot C+a^2\xi \right).
\end{align}
Using \eqref{4.8} and \eqref{4.9} we get
\begin{align}\label{4.111}
\nabla _{\dot C}\dot C=\bar k_1\overline W=-a\varphi \dot C .
\end{align}
Substituting the latter equality in \eqref{4.11}, we obtain
\begin{align*}
\widetilde \nabla _{C^\prime }C^\prime =-\frac{1}{a}\varphi \dot C+\xi .
\end{align*}
Denoting $E_2(\widetilde s)=-\frac{1}{a}\varphi \dot C+\xi $, we check directly that $\widetilde g(E_2,E_2)=1$ and $\widetilde g(E_1,E_2)=0$. Then for the curvature $\widetilde k(s)$ we find
\begin{align*}
\widetilde k(\widetilde s)=\vert \widetilde \nabla _{C^\prime }C^\prime \vert =\sqrt{\displaystyle\vert \widetilde g(\widetilde \nabla _{C^\prime }C^\prime ,\widetilde \nabla _{C^\prime }C^\prime )\vert }=\sqrt{\displaystyle\vert \widetilde g(E_2,E_2)\vert }=1.
\end{align*}
Thus we have
\begin{align}\label{4.13}
\widetilde \nabla _{C^\prime }C^\prime =\widetilde kE_2 .
\end{align}
Now, we compute
\begin{align}\label{4.14}
\widetilde \nabla _{C^\prime }E_2=\frac{1}{\vert a\vert }\widetilde \nabla _{\dot C}\left(-\frac{1}{a}\varphi \dot C+\xi \right)=\frac{1}{\vert a\vert }\left(-\frac{1}{a}\widetilde \nabla _{\dot C}\varphi \dot C+\widetilde \nabla _{\dot C}\xi \right).
\end{align}
Further, by using \eqref{2.4} we get 
\begin{align}\label{4.15}
\widetilde \nabla _{\dot C}\varphi \dot C=\nabla _{\dot C}\varphi \dot C-
a^2\xi .
\end{align}
From the well known formula $(\nabla _{\dot C}\varphi )\dot C=\nabla _{\dot C}\varphi \dot C-
\varphi (\nabla _{\dot C}\dot C)$ we express
\begin{align}\label{4.16}
\nabla _{\dot C}\varphi \dot C=(\nabla _{\dot C}\varphi )\dot C+\varphi (\nabla _{\dot C}\dot C) .
\end{align}
By virtue of \eqref{2.2} we find
\begin{align*}
(\nabla _{\dot C}\varphi )\dot C=a(a\xi +\varphi ^2\dot C) .
\end{align*}
Taking into account \eqref{4.111} we have
\begin{align*}
\varphi (\nabla _{\dot C}\dot C)=-a\varphi ^2\dot C .
\end{align*}
Substituting the latter two equalities in \eqref{4.16} we obtain 
\begin{align}\label{4.17}
\nabla _{\dot C}\varphi \dot C=a^2\xi .
\end{align}
From \eqref{4.15} and \eqref{4.17} it follows
\begin{align}\label{4.18}
\widetilde \nabla _{\dot C}\varphi \dot C=0 .
\end{align}
By using \eqref{2.4} and \eqref{2.3} we get
\begin{align}\label{4.19}
\widetilde \nabla _{\dot C}\xi =\nabla _{\dot C}\xi =-\varphi \dot C .
\end{align}
Substituting \eqref{4.18} and \eqref{4.19} in \eqref{4.14} we receive
$\widetilde \nabla _{C^\prime }E_2=-\frac{1}{\vert a\vert }\varphi \dot C$. We rewrite the last equality in the following equivalent form 
\begin{align*}
\widetilde \nabla _{C^\prime }E_2=-\widetilde kE_1+\widetilde kE_1-\widetilde k\frac{1}{\vert a\vert }\varphi \dot C=-\widetilde kE_1-\widetilde k\left(\frac{1}{\vert a\vert }\varphi \dot C-E_1\right)
\end{align*}
and put $E_3(\widetilde s)=\frac{1}{\vert a\vert }\varphi \dot C-E_1=\frac{1}{\vert a\vert }(\varphi \dot C-\dot C)$. Immediately we verify that $\widetilde g(E_3,E_3)=-1$, $\widetilde g(E_1,E_3)=\widetilde g(E_2,E_3)=0$. Now, we obtain
\begin{align}\label{4.20}
\widetilde \nabla _{C^\prime }E_2=-\widetilde kE_1-\widetilde \tau E_3 ,
\end{align}
where $\widetilde \tau =\widetilde k$. Finally, we have
\begin{align*}
\widetilde \nabla _{C^\prime }E_3=\frac{1}{\vert a\vert }\widetilde \nabla _{\dot C}\frac{1}{\vert a\vert }(\varphi \dot C-\dot C)=\frac{1}{a^2}\left(\widetilde \nabla _{\dot C}\varphi \dot C-\widetilde \nabla _{\dot C}\dot C\right).
\end{align*}
Taking into account \eqref{4.13} and \eqref{4.18}, we infer 
\begin{align}\label{4.21}
\widetilde \nabla _{C^\prime }E_3=-\widetilde \tau E_2 .
\end{align}
The equalities \eqref{4.13}, \eqref{4.20} and \eqref{4.21} show that $C(\widetilde s)$ is a Frenet curve of osculating order 3. Note that in our case $\epsilon _1=\epsilon _2=-\epsilon _3=1$.\\
(ii) In (i) we obtained that $\widetilde k=\widetilde \tau =1$ which means that the Frenet slant curve $C$ with respect to the metric $\widetilde g$ is a proper helix.
\end{proof}
\section{Examples}\label{sec-6}
Let us consider ${\R}^2$ endowed with an almost complex structure $J$ and a  metric $h$  which are defined with respect to the local basis $\left\{\frac{\partial}{\partial x},\frac{\partial}{\partial y}\right\}$  as follows: 
\begin{equation*}
\begin{array}{ll}
\displaystyle
J\left(\frac{\partial}{\partial x}\right)=\frac{\partial}{\partial y}, \quad J\left(\frac{\partial}{\partial y}\right)=-\frac{\partial}{\partial x}, \\
\displaystyle
h\left(\frac{\partial}{\partial x},\frac{\partial}{\partial x}\right)=-h\left(\frac{\partial}{\partial y},\frac{\partial}{\partial y}\right)=-1 , \quad
h\left(\frac{\partial}{\partial x},\frac{\partial}{\partial y}\right)=0 .
\end{array}
\end{equation*}
The almost complex structure $J$ acts as an anti-isometry with respect to the metric $h$, i.e. $h(JX,JY)=-h(X,Y)$ and $h$ is a pseudo-Riemannian neutral metric (called Norden metric).
The associated neutral metric $\widetilde h$ is defined by $\widetilde h(X,Y)=h(X,JY)$. It is easy to check that the almost complex structure $J$ is parallel with respect to the Levi-Civita connection $\nabla ^h$ of the metric $h$, i.e. $\nabla ^hJ=0$. Hence, $N^2=({\R}^2,J,h)$ is a K\"ahler-Norden manifold. Then, according to \cite[Theorem 3.5]{IMM}, the product manifold $M^3=\R^+\times N^2$ equipped with the almost contact B-metric structure $(\varphi ,\xi ,\eta ,g)$ given by
\[
\xi =\frac{\partial}{\partial z}, \, \, \, \eta ={\rm d}z, \quad \varphi _{\vert N^2}=J, \, \, \, \varphi \frac{\partial}{\partial z}=0, \quad g={\rm d}z^2+\cos 2z \, h-\sin 2z \, \widetilde h
\]
(${\rm d}z$ is the coordinate 1-form on $\R^+$) is a Sasaki-like almost contact B-metric manifold. The vector fields
\[
e_1=\cos z\frac{\partial}{\partial x}+\sin z\frac{\partial}{\partial y}, \quad 
e_2=-\sin z\frac{\partial}{\partial x}+\cos z\frac{\partial}{\partial y}, \quad e_3=\frac{\partial}{\partial z}
\]
are linearly independent at each point of $M^3$ and they satisfy
\begin{equation}\label{6.1}
g(e_1,e_1)=-g(e_2,e_2)=g(e_3,e_3)=1, \, \, g(e_i,e_j)=0, i\neq j\in \{1,2,3\},
\end{equation}
\begin{equation}\label{6.2}
\left[e_1,e_2\right]=0, \quad \left[e_1,e_3\right]=-e_2, \quad \left[e_2,e_3\right]=e_1.
\end{equation}
Now, by using the Koszul formula
\begin{align}\label{5.2}
2g(\nabla _{e_i}e_j,e_k)=g([e_i,e_j],e_k)+g([e_k,e_i],e_j)+g([e_k,e_j],e_i)
\end{align}
we find the components of the Levi-Civita connection $\nabla $ of the metric $g$. The non-zero ones of them are
\begin{align}\label{5.5}
\nabla _{e_1}e_2=\nabla _{e_2}e_1=-e_3 , \quad \nabla _{e_1}e_3 =-e_2, \quad 
\nabla _{e_2}e_3=e_1.
\end{align}
With the help of \eqref{2.4} and \eqref{5.5} we obtain that the non-zero components of the Levi-Civita connection $\widetilde \nabla $ of the metric $\widetilde g$ are
\begin{align}\label{6.3}
\widetilde \nabla _{e_1}e_1=-\widetilde \nabla _{e_2}e_2=-e_3 , \quad \widetilde \nabla _{e_1}e_3 
=-e_2, \quad \widetilde \nabla _{e_2}e_3=e_1.
\end{align}
We consider the following curves in $M^3=\R^+\times N^2$ with respect to the basis 
$\{e_1,e_2,e_3\}$ of $TM^3$:
\par
{\bf (a)} $C(t)=(\cosh t,\sinh t,t)$, $t\in \R$. We find $\dot C=(\sinh t,\cosh t,1)$ and
$\varphi \dot C=(-\cosh t,\sinh t,0)$. Then $g(\dot C,\dot C)=0$, $a=\eta (\dot C)=1$, $b=g(\dot C,\varphi \dot C)=-\sinh 2t$, that is $C$ is a slant null curve such that the function $b$ is not a constant. By using \eqref{5.5} we obtain
\begin{align}\label{6.4}
\nabla _{\dot C}\dot C=2\cosh t e_1-\sinh 2t e_3 .
\end{align}
First, we take $S(TC^\bot )$ spanned by $W=\displaystyle\left(-\frac{\cosh t}{\cosh 2t},\frac{\sinh t}{\cosh 2t},\tanh 2t\right)$. Then the unique $N$ corresponding to $W$ is  given by\\
\[
N=\left(\frac{\sinh t(4\cosh ^2t-1)}{2\cosh ^22t},\frac{-\cosh t(4\sinh ^2t+1)}{2\cosh ^22t},
\frac{1}{2\cosh ^22t}\right).
\]
Now, \eqref{6.4} becomes
\[
\nabla _{\dot C}\dot C=\tanh 2t \dot C-2\cosh^2t W .
\]
By using \eqref{5.5} we obtain
\[
\nabla _{\dot C}N=-\frac{3+\cosh2t}{2\cosh^22t} W, \quad  \nabla _{\dot C}W=\frac{3+\cosh2t}{2\cosh^22t} \dot C+2\cosh ^2t N.
\]
Comparing the above three equations with \eqref{2.7}, \eqref{2.8} and \eqref{2.9} we get
\begin{equation*}
h=\tanh 2t , \qquad k_1=-2\cosh^2t , \qquad k_2=-\frac{3+\cosh2t}{2\cosh^22t}
\end{equation*}
with respect to the Frenet frame ${\bf F}=\{\dot {C}, N, W\}$. \\
Next, consider the Frenet frame
${\bf {\overline F}}=\{\dot {C}, \overline N, \overline W\}$, where
\begin{equation*} 
\overline W=\left(\frac{1}{\cosh t},0,-\tanh t\right), \quad \overline N\left(\frac{\sinh t}{2\cosh ^2t},
-\frac{1}{2\cosh t},\frac{1}{2\cosh ^2t}\right)
\end{equation*} 
we have $\nabla _{\dot C}\dot C=2\cosh ^2t \overline W$ and $\nabla _{\dot C}\overline N=\displaystyle\frac{1}{\cosh ^2t} \overline W$. Thus $\overline h=0$, \\
$\overline k_1=2\cosh ^2t $, $\overline k_2=\displaystyle\frac{1}{\cosh ^2t}$ with respect to the Frenet frame ${\bf {\overline F}}=\{\dot {C}, \overline N, \overline W\}$.
This example is an illustration of the results in Proposition \ref{Proposition 3.7}.
\par 
{\bf (b)} $C(t)=(C^*,t,-t)$, $t\in \R$, $C^*\in \R$ is a slant null curve for which $a=-1$ and $b=0$. 
From Proposition \ref{Proposition 3.14} it follows that $C$ is a $\varphi $-geodesic. Really, by virtue of \eqref{5.5} we find $\nabla _{\dot C}\dot C=-e_1=\varphi \dot C$. The frame
${\bf \overline F}=\{\dot C=(0,1,-1), \, \overline N=(0,-1/2,-1/2), \, \overline W=(-1,0,0)\}$ is a Cartan Frenet frame for $C$ and $\bar k_2=1/2$. Hence, $C$ is a generalized helix which confirms the assertion in Corollary \ref{Corollary  3.15}.\\
Now, we consider the curve $C(t)=(C^*,t,-t)$ with respect to the metric $\widetilde g$. 
By direct computations we obtain $\widetilde g(\dot C,\dot C)=1$ and $\widetilde \eta (\dot C)=-1$. Hence $C(t)$ is a unit speed spacelike curve and it is also a slant curve. By using \eqref{6.3} we get
$\widetilde \nabla _{\dot C}\dot C=e_3-e_1$. The vector fields $E_1=(0,1,-1)$, $E_2=(-1,0,1)$, 
$E_3=(-1,-1,1)$ are orthonormal with respect to $\widetilde g$ and $\widetilde g(E_1,E_1)=\widetilde g(E_2,E_2)=-\widetilde g(E_3,E_3)=1$. Hence, $C$ is of osculating order 3 and $\widetilde k=\widetilde \tau =1$. The obtained results for $C$ with respect to $\widetilde g$ concur with those in  Theorem \ref{Theorem 4.2}.
\par
{\bf (c)} $C(t)=(\sqrt{2}t/2, -\sqrt{2}t/2,C^*)$, $t\in \R$, $C^*\in \R$ is a Legendre null curve for which
$b=1$ and $\nabla _{\dot C}\dot C=e_3=\xi $. The frame
${\bf \overline F}=\{\dot C=(\sqrt{2}/2,-\sqrt{2}/2,0), \\ \overline N=(\sqrt{2}/2,\sqrt{2}/2,0), \, \overline W=(0,0,1)\}$ is a Cartan Frenet frame for $C$ and $\bar k_2=0$. Hence $C$ is a null cubic which confirms (iii) in Proposition \ref{Proposition 3.20}.\\
With respect to $\widetilde g$ we have $\widetilde g(\dot C,\dot C)=1$ and $\widetilde \eta (\dot C)=0$ that is $C$ is a unit speed spacelike Legendre curve. Moreover, by virtue of \eqref{6.3} we obtain
$\widetilde \nabla _{\dot C}\dot C=0$, i.e. $C$ is a geodesic.

Consider a 3-dimensional solvable Lie group $G$  with a Lie algebra $\mathfrak {g}$. Let  
$\mathfrak {g}$  be determined by the commutators of the basis  $\{e_1, e_2, e_3\}$ of left invariant vector fields given by \eqref{6.2}.
In \cite{IMM} were defined the almost contact structure $(\varphi, \xi, \eta )$ and the left invariant B-metric $g$ on $G$ by
\begin{align*}
\begin{array}{llll}
\varphi e_1=e_2, \, \,  \varphi e_2=-e_1, \, \, \varphi e_3=0, \quad \xi =e_3, \, \, \eta (e_3)=1, 
\eta (e_1)=\eta (e_2)=0,\\
g(e_1,e_1)=-g(e_2,e_2)=g(e_3,e_3)=1, \quad
g(e_i,e_j)=0, \, i\neq j \in\{1,2,3\}
\end{array}
\end{align*}
and it was shown that $(G,\varphi,\xi,\eta,g)$ is a Sasaki-like almost contact B-metric manifold. 
Since $e_1,e_2,e_3$ satisfy \eqref{6.1} and \eqref{6.2}, the non-zero components of the Levi-Civita connection $\nabla $ of the metric $g$ are given by \eqref{5.5}.

Now, we construct slant null curves in the Sasaki-like almost contact B-metric manifold $(G,\varphi,\xi,\eta,g)$.
Consider the curve $C(t)=e^{tX}$ on $G$, where
$t\in {\R}$ and $X\in \mathfrak {g}$. Hence the tangent vector to $C(t)$ at the identity element $e$ of 
$G$ is $\dot {C}(0)=X$. Let the coordinates $(p,q,r)$ of $\dot C(0)$ with respect to the basis $\{e_1, e_2, e_3\}$ are given by
\begin{align}\label{5.6}
p=-\epsilon \sqrt{\frac{\sqrt{a^4+b^2}-a^2}{2}}, \quad q=\sqrt{\frac{\sqrt{a^4+b^2}+a^2}{2}},
\quad r=a,
\end{align}
where $a, b\in \R$, $(a,b)\neq (0,0)$ and $\epsilon =\{{\rm sign}\, b\}=\{\pm 1\}$. Since
$g(\dot {C},\dot {C})=0$ and $\eta (\dot C)=a$, $C(t)$ is a slant null curve in $(G,\varphi,\xi,\eta,g)$. Also, having in mind that $\varphi \dot C=(-q,p,0)$,
we have $g(\dot C,\varphi \dot C)=b$.  Furthermore, using \eqref{5.5}, one obtains
\begin{align}\label{5.7}
\begin{array}{lll}
\nabla _{\dot C}{\dot C}=aqe_1-ape_2+be_3 \\
\qquad \, \, \, =\displaystyle\sqrt{a^4+b^2}\left(\frac{aq}{\sqrt{a^4+b^2}}e_1-\frac{ap}{\sqrt{a^4+b^2}}e_2+\frac{b}{\sqrt{a^4+b^2}}e_3\right)\\
\qquad \, \, \, =\sqrt{a^4+b^2} \, \overline W,
\end{array}
\end{align}
where the vector field
\begin{align}\label{W_1}
\overline W=\left(\frac{aq}{\sqrt{a^4+b^2}},-\frac{ap}{\sqrt{a^4+b^2}},\frac{b}{\sqrt{a^4+b^2}}\right)
\end{align}
is a spacelike unit. 
Then the unique $\overline N$ corresponding to $\overline W$  is given by
\begin{align}\label{N_1}
\overline N=\left(-\frac{a^2p+2bq}{2(a^4+b^2)},\frac{-a^2q+2bp}{2(a^4+b^2)},\frac{a^3}{2(a^4+b^2)}\right).
\end{align}
Thus, by using \eqref{5.5}, we obtain
\begin{align}\label{5.8}
\begin{array}{ll}
\nabla_{\dot C}\overline N=\displaystyle\frac{a^2}{2\sqrt{a^4+b^2}}\overline W, \\
\nabla_{\dot C}\overline W=-\displaystyle\frac{a^2}{2\sqrt{a^4+b^2}}\dot C-
\sqrt{a^4+b^2}\, \overline N .
\end{array}
\end{align}
Comparing the equations \eqref{5.7} and \eqref{5.8} with \eqref{2.7} and \eqref{2.8}, \eqref{2.9}, respectively, we get
\begin{align}\label{5.88}
\bar h=0 , \qquad \bar k_1=\sqrt{a^4+b^2}, \qquad \bar k_2=\displaystyle\frac{a^2}
{2\sqrt{a^4+b^2}} 
\end{align}
with respect to to the Frenet frame ${\bf \overline F}=\{\dot C, \overline N, \overline W\}$.\\
Further, we find the matrix representation of $C(t)$ and ${\bf \overline F}$. 
Let us recall that the adjoint representation $\rm Ad$ of $G$ is the following Lie group homomorphism
\[
\rm {Ad} : G \longrightarrow Aut({\g}).
\]
For $X\in {\g}$, the map ${\rm {ad}}_X : {\g}\longrightarrow {\g}$ is defined by ${\rm {ad}}_X(Y)=[X,Y]$, where by ${\rm ad}_X$ is denoted  ${\rm {ad}}(X)$. Due to the Jacobi identity, the map
\[
\rm {ad} : {\g} \longrightarrow End({\g}) : X\longrightarrow ad_X
\]
is Lie algebra homomorphism, which is called  adjoint representation of ${\g}$.
Since the set ${\rm End}({\g})$ of all ${\K}$-linear maps from ${\g}$ to ${\g}$ is isomorphic to the set of all  $(n\times n)$ matrices ${\rm M}(n,{\K})$ with entries in ${\K}$, $\rm {ad}$ is a matrix representation of ${\g}$. We denote by $M_i$ the matrices of ${\rm ad}_{E_i}$ (i=1,2,3) with respect to the basis $\{e_1,e_2,e_3\}$ of ${\g}$.  Then for an arbitrary $X=x_1e_1+x_2e_2+x_3e_3$ ($x_1, x_2, x_3 \in {\R}$) in ${\g}$ the matrix $A$ of ${\rm ad}_X$ is $A=x_1M_1+x_2M_2+x_3M_3$.
Then by virtue of the well known identity $e^A={\rm {Ad}}\left(e^X\right)$ we find the matrix representation of the Lie group $G$. By using \eqref{6.2} we obtain $M_1$, $M_2$, $M_3$ and then $A$
\[
M_1=\left(\begin{array}{llr}
0 & 0 & 0 \cr
0 & 0 & -1\cr
0 & 0 & 0
\end{array}\right), \quad
M_2=\left(\begin{array}{lcr}
0 & 0 & 1 \cr
0 & 0 & 0 \cr
0 & 0 & 0
\end{array}\right), \quad
M_3=\left(\begin{array}{lrl}
0 & -1 & 0 \cr
1 & 0  & 0 \cr
0 & 0 & 0
\end{array}\right),
\]
\begin{align}\label{5.9}
A=\left(\begin{array}{crr}
0 & -x_3 & x_2 \cr
x_3 & 0  & -x_1  \cr
0 & 0 & 0
\end{array}\right).
\end{align}
The characteristic polynomial of A is
\[
P_A(\lambda )=-\lambda(\lambda ^2+x_3^2) =0 .
\]
Hence for the eigenvalues $\lambda _i \, (i = 1, 2, 3)$ of $A$ we have
\[
\lambda _1=0 , \quad \lambda _2=ix_3 , \quad \lambda _3=-ix_3 , \quad i^2=-1.
\]
By the assumption that $x_3\neq 0$, the eigenvectors
\[
p_1=(x_1,x_2,x_3), \quad p_2=(1,i,0),\quad p_3=(i,1,0)
\]
corresponding to $\lambda _1, \lambda _2, \lambda _3$, respectively, are linearly independent for arbitrary $x_1$, $x_2$ and $x_3\neq 0$. For the change of basis matrix P and its inverse matrix 
$P^{-1}$ we get
\[
P=\left(\begin{array}{rll}
x_1 & 1 & i \cr
x_2 & i & 1 \cr
x_3 & 0 & 0
\end{array}\right) , \quad 
P^{-1}=\frac{1}{2x_3}\left(\begin{array}{ccc}
0 & 0 & 2 \cr
x_3 & -ix_3 & -x_1+ix_2 \cr
-ix_3 & x_3 & -x_2+ix_1
\end{array}\right) .
\]
By using that $e^A = Pe^JP^{-1}$, where $J$ is the diagonal matrix with elements 
$J_{ii} =\lambda _i$, we obtain the matrix representation of the Lie group $G$ in case  $x_3\neq 0$
\begin{align}\label{5.10}
\small G=\left\{e^A=
\left(\begin{array}{crc}
\cos x_3 & -\sin x_3  & \frac{x_1}{x_3}(1-\cos x_3)+\frac{x_2}{x_3}\sin x_3  \cr \cr
\sin x_3 & \cos x_3  & \frac{x_2}{x_3}(1-\cos \alpha x_3)-\frac{x_1}{x_3}\sin x_3 \cr \cr
0 & 0 & 1
\end{array}\right)
\right\} .
\end{align}
The coordinates of the vector field $t\dot C\in {\g}$, $t\in \R$, are $(tp,tq,ta)$, where $p,q$ are given by \eqref{5.6} and $a\neq 0$. Since ${\rm Ad}(C(t))={\rm Ad}\left(e^{t\dot c}\right)$,
we find ${\rm Ad}(C(t))$ replacing $x_1$, $x_2$ and $x_3$ in \eqref{5.10} with
$tp$, $tq$ and $ta$, respectively. Thus, for the matrix representation of a slant null curve 
$C(t)$, which is not a Legendre curve, we have
\begin{align}\label{5.11}
{\rm Ad}(C(t))=
\left(\begin{array}{ccc}
\cos at & -\sin at  & \frac{p}{a}(1-\cos at)+\frac{q}{a}\sin at  \cr \cr
\sin at & \cos at  & \frac{q}{a}(1-\cos at)-\frac{p}{a}\sin at \cr \cr
0 & 0 & 1
\end{array}\right) .
\end{align}
We may obtain the  matrix representations of $\dot C$, $\overline W$ and $\overline N$ replacing $x_1$, $x_2$ and $x_3$ in \eqref{5.9} with their coordinates, determined by   \eqref{5.6},  \eqref{W_1} and  \eqref{N_1}, respectively.
\\
Now, by using Corollary \ref{Corollary 3.12} we will find a slant null curve in the Lie group $(G,\varphi,\xi,\eta,g)$ which is a generalized helix. Take $b=\epsilon \sqrt{1-a^4}$ in \eqref{5.6}, \eqref{W_1}, \eqref{N_1} and \eqref{5.88}, where 
$\epsilon =\{{\rm sign}\, b\}=\{\pm 1\}$ and  $a\in [-1,0)\cup (0,1]$. Then the slant null curve 
$C_1(t)$ has a tangent vector $\dot C_1(0)$ with coordinates $p_1,q_1,r_1$ given by 
\begin{align}\label{5.12}
p_1=-\epsilon \frac{\sqrt{1-a^2}}{\sqrt{2}}, \quad q_1=\frac{\sqrt{1+a^2}}{\sqrt{2}}, \quad r_1=a
\end{align}
and curvatures $\bar k_1=1$, $\bar k_2=\frac{a^2}{2}$. From Corollary \ref{Corollary 3.12} it follows that the slant null curve $C_1(t)$ is a generalized helix with respect to the Frenet frame ${\bf F_1}=\{\dot C_1, N_1, W_1\}$, where
\begin{align}\label{5.13}
\begin{array}{ll}
W_1=\left(\frac{a\sqrt{1+a^2}}{\sqrt{2}}, \epsilon \frac{a\sqrt{1-a^2}}{\sqrt{2}}, \epsilon \sqrt{1-a^4}\right),\\ \\
N_1=\left(-\epsilon \frac{\sqrt{1-a^2}(2+a^2)}{2\sqrt{2}}, -\frac{\sqrt{1+a^2}(2-a^2)}{2\sqrt{2}}, \frac{a^3}{2}\right) .
\end{array}
\end{align}
We obtain the matrix representation of $C_1(t)$ replacing $p$ and $q$ in \eqref{5.11} with $p_1$ and 
$q_1$, respectively, determined by \eqref{5.12}.

\end{document}